\documentclass{amsart}
\usepackage{macros}
\usepackage{dsfont}
\usepackage[left=1.2in,right=1.2in,bottom=1in]{geometry}
\upsymbols{Tor,Isom,Div,Hom,Br,inv,WC,nr,red,GL,SL,Aut,tors,Hom,Sw,Jac}
\upsymbols{Gal,res,im,coker,Frob,rk,Res,Ext,disc,Id,FS,cyc,Irr,sp,wild,tr}
\bbsymbols{A,G,Z,Q,R,F,C,E}
\scrsymbols{F,C,S,L,M}
\calsymbols{O,T,F,G,C}
\fraksymbols{p,P,N,d}
\newcommand{\triv}{\mathds{1}}

\newcommand{\epsg}{\aG}
\newcommand{\delg}{\Delta^{\cG}}

\newcommand{\aG}{a^{\cG}}
\newcommand{\finplaces}{\places^0}

\title{Conductors of Twisted Weil--Deligne Representations}
\author{Matthew Bisatt and Ross Paterson}
\makeatletter
\hypersetup{
pdftitle={\@title},
pdfauthor={Matthew Bisatt and Ross Paterson},
pdfsubject={11G05 (11G10 11N45 14H52)},
pdfkeywords={abelian varieties, number theory, L-functions, conductors, Weil--Deligne representations, Langlands, Rankin--Selberg}
}
\makeatother
\address{School of Mathematics, Fry Building, University of Bristol, Bristol, BS8 1UG.}
\email{matthew.bisatt@bristol.ac.uk}
\email{rosspatersonmath@gmail.com}

\begin{document}
\begin{abstract}
We study the behaviour of conductors of $L$-functions associated to certain Weil--Deligne representations under twisting.  For each global field $K$ we prove a sharp upper bound for the conductor of the Rankin--Selberg $L$-function $L(A\boxtimes B,s)$ where $A,B/K$ are abelian varieties.
\end{abstract}
\maketitle
\section{Introduction}

$L$-functions are a number theoretic orchestra; an ensemble of algebra, geometry, and analysis.  One useful tool in this area is the ability to build new (conjecturally automorphic) $L$-functions from old $L$-functions through twisting.  In particular, given a global field $K$ and two abelian varieties $A,B/K$, one can construct the Rankin--Selberg convolution $L(A\boxtimes B,s)$.  There is significant motivation to study these $L$-functions already in the literature: the Rankin--Selberg method, together with the converse theorem, is one of the main lines of attack in the Langlands program \cites{MR2275908,MR3307911,MR1728875,MR1395406}.

Our objective is to study certain local invariants associated to such $L$-functions, namely: the Artin conductor exponent and the Swan conductor exponent\footnote{After all, every good orchestra needs a conductor.}.  In order to study these, it will be necessary to understand the associated local Weil--Deligne representations.  On this level we must consider the tensor product of representations corresponding to abelian varieties, since this corresponds to $A\boxtimes B$ via the local Langlands correspondence.  Our main result is the following.
\begin{theorem}[\Cref{thm:tensors of abvars}]\label{INTRO:thm:cond of tensor of abvars}
Let $A$ and $B$ be abelian varieties over a non-Archimedean local field $F$ of residue characteristic $p$, and write $\rho_A,\rho_B$ for the associated $\ell$-adic representations, and $\deg(A),\deg(B)$ for the degrees of the associated $L$-polynomials.  Define the constant
    \[C_p(A,B):=\min\bigg\{\max\set{2,p-1},\ \abs{\rho_A} - \deg(A),\ \abs{\rho_B}-\deg(B)\bigg\}.\]
Then the Artin conductor exponent of the associated $\ell$-adic representations $\rho_A$ and $\rho_B$ satisfies
    \[a(\rho_A \otimes \rho_B)\leq  \dim(\rho_A)a(\rho_B) + \dim(\rho_B)a(\rho_A) - C_p(A,B)\min \set{a(\rho_A), a(\rho_B)} - \bigg(\deg(A\boxtimes B)-\deg(A)\deg(B)\bigg).\]
    Moreover, if $a(\rho_A)\leq 1$ then in fact
    \[a(\rho_A \otimes \rho_B) = \dim(\rho_A)a(\rho_B) + \deg(B)a(\rho_A) - \bigg(\deg(A\boxtimes B)-\deg(A)\deg(B)\bigg).\]
\end{theorem}

We obtain this by splitting the conductor exponent into a tame part and a wild part (i.e. the Swan conductor).  For this latter piece, we obtain the following bound.

\begin{theorem}[\Cref{thm:swancond of tensor of abvars}]\label{INTRO:thm:swancond of tensor of abvars}
    Let $A$ and $B$ be abelian varieties over a non-Archimedean local field $F$ of residue characteristic $p$. Then
\[\Sw(\rho_A \otimes \rho_B)\leq \dim(\rho_A)\Sw(\rho_B) + \dim(\rho_B)\Sw(\rho_A) - \max\set{2, p-1}\min \set{\Sw(\rho_A), \Sw(\rho_B)}.\]
\end{theorem}

\begin{rem}
The inequalities of \Cref{INTRO:thm:cond of tensor of abvars,INTRO:thm:swancond of tensor of abvars} are trivially sharp: if $A$ has good reduction, then equality holds as $a(\rho_A \otimes \rho_B)=\dim(\rho_A)a(\rho_B)$ with the remaining terms vanishing (and similarly for $\Sw$). A family of nontrivial examples of sharpness, one for each odd prime $p$, is given in \Cref{subsec:sharpness}.
\end{rem}

A corollary of our work above is the following global result, which uses a simplified version of the above bound, and answers a question raised to us by Andrew Booker.

\begin{theorem}[\Cref{cor:Cond of TP}]\label{INTRO:thm:Cond of TP}
    Let $A,B$ be abelian varieties over a global field $K$ of conductors $\fN_A$ and $\fN_B$. Define
    \[\fd(A,B):=\prod_{\substack{\fp\textnormal{ prime}\\v_\fp(\fN_A)v_\fp(\fN_B)>1}}\fp^{\deg_\fp(A\boxtimes B)-\deg_\fp(A)\deg_\fp(B)}.\]
    Then 
    \[N(A\boxtimes B)\hspace{2pt}\Bigg\vert\hspace{2pt} \frac{\fN_A^{2\dim B}\fN_B^{2\dim A}}{\fd(A,B)\gcd(\fN_A,\fN_B)^2},\]
    where $N(A\boxtimes B)$ is the conductor of the Rankin--Selberg $L$-function $L(A\boxtimes B, s)$, and $\deg_\fp$ is the degree of the corresponding $L$-factor.
\end{theorem}

The machinery we have developed here requires little arithmetic input.  Hence, most of our results are shown on a purely group-theoretic level in \Cref{sec:group conductors}, in order to make them more flexible for application and in case they are of independent interest.

\subsection*{Comparison to current literature}
There has been significant interest in understanding conductor exponents of tensor products of Weil--Deligne representations.  In the setting of our results, there is the (sharp) bound of Bushnell--Henniart, which we recall below.
\begin{theorem}[\cite{MR4404782}*{Theorem B.2}]%
Let $A$ and $B$ be abelian varieties over a non-Archimedean local field $F$, with associated Weil--Deligne representations $\rho_A,\rho_B$.  Then
\[a(\rho_A \otimes \rho_B) \leq \dim(\rho_A)a(\rho_B) + \dim(\rho_B)a(\rho_A) - 2\min\set{a(\rho_A), a(\rho_B)},\]
where $a(\cdot)$ is the Artin conductor exponent.
\end{theorem}
\Cref{INTRO:thm:swancond of tensor of abvars} is an improvement on this result.  Our techniques indicate that the factor of $2$ multiplying the minimum arises essentially because the representations are (essentially) symplectic.  Meanwhile our factor of $p-1$ arises as a consequence of Weil's theorem on the rationality of the characteristic polynomials of these representations.

Our methods and results are more general than the setting presented above:  we prove analogous statements for (local) Weil--Deligne representations satisfying certain technical hypotheses.  Throughout, by a (local) Weil--Deligne representation, we implicitly mean a complex Frobenius-semisimple one (equivalently an admissible representation for the Langlands correspondence).  In this general setting, Bushnell--Henniart \cite{BushHen2}*{Theorem C} prove that for two general Weil--Deligne representations $\rho_1,\rho_2$ one only saves one factor of the minimum:
\[a(\rho_1 \otimes \rho_2) \leq \dim(\rho_1)a(\rho_2) + \dim(\rho_2)a(\rho_1)- \min\set{a(\rho_1), a(\rho_2)}.\]

Our approach suggests that if one imposes restrictions on the corresponding Weil--Deligne representations then a better constant, in a similar vein to the results on abelian varieties, in front of the minimum term should be possible. 
Indeed, if $\rho_j = \oplus_{n \geq 1}\ \sigma_n^{(j)} \otimes \sp(n)$ for $j=1,2$, then we expect said constant in the general case to depend on the smallest integer $n$ for which some $\sigma_{n}^{(j)}\neq 0$.  If, moreover, one knows additional representation theoretic information about the $\sigma_n$ (e.g. essentially symplectic, or rationality properties), then such bounds should be able to be improved further.

\subsection*{Acknowledgements} We thank Andrew Booker for bringing the questions addressed in this article to our attention and for valuable comments, and Besfort Shala for helpful conversations.  We are grateful to Guy Henniart for bringing the work of Bushnell--Henniart in \cite{MR4404782}*{Appendix B} to our attention. 

\newpage
\subsection*{Notation} For each real number $M$, we denote by $\RR_{\geq M}$ the set of real numbers greater than or equal to $M$.  For any subset $S\subset \RR$ we write $S_{\geq M}:=S\cap \RR_{\geq M}$.  Below we record some common notation found in this article.

\vspace{-5pt}
\subsubsection*{Finite Group Representations}  In \Cref{sec:group conductors} we work with arbitrary finite groups.  For every finite group $G$ and finite dimensional $\CC$-representations $\tau,\sigma$ of $G$, we denote by:

\begin{tabular}{cl}
$\abs{\tau}$ & the dimension of $\tau$; \\
$\tau^*$& the contragredient representation of $\tau$; \\
$\gp{\tau,\sigma}_{G}$& the usual inner product of $\tau$ and $\sigma$ as representations of $G$;\\
$\triv$ & the trivial representation of $G$;\\
$\tau^G$ & the $G$-fixed elements under $\tau$.
\end{tabular}

\subsubsection*{Local Arithmetic} In the local arithmetic setting (\Cref{subsec:WD reps background}--\ref{subsec:WD reps conductor exponents for tps}), typically we use the following notation:

\begin{tabular}{cl}
$F$ & a non-Archimedean local field;\\
$\tau,\sigma$ & Weil representations over $F$;\\
$\rho$ & a Weil--Deligne representation over $F$;\\
$A/F$ & an abelian variety;\\
$\bar{F}$& a (fixed once and for all) algebraic closure of $F$;\\
$F^{\nr}$ & $\subseteq \bar{F}$ the maximal unramified extension of $F$\\
$G_F$&the absolute Galois group $\gal(\bar{F}/F)$;\\
$a(\rho)$ & the Artin conductor exponent;\\
$\Sw(\rho)$ & the Swan conductor exponent; \\
$\deg(\rho)$ & the degree of the $L$-polynomial of $\rho$;\\
$\sp(n)$ & the special representation of dimension $n$;\\
$\rho_A$&the associated Weil--Deligne representation of $A$.
\end{tabular}

\subsubsection*{Global Arithmetic}  In \Cref{subsec:global results} we work over a global field $K$, and use the following notation:

\begin{tabular}{cl}
$\finplaces_K$ & the set of finite places of $K$;\\
$\fp$&$\in \finplaces_K$ a prime of $K$;\\
$v_\fp$&the normalised valuation at $\fp$ on $K$;\\
$A,B$ & abelian varieties over $K$;\\
$\deg_\fp(A),\, \deg_{\fp}(A\boxtimes B)$ & the degrees of the corresponding $L$-factors at $\fp$.
\end{tabular}

\vspace*{5pt}


\section{Conductors for group filtrations}\label{sec:group conductors}
Here we abstract the notion of the Artin conductor exponent of a representation to being a number associated to a sequence of subgroups and a representation.  We then prove statements about this general object which require no arithmetic input.  These results will find arithmetic applications in \Cref{sec:arithmetic}.
\subsection{Definitions}
We begin by introducing the central objects for this section.
\begin{definition}
    Let $\cG=(G_i)_{i\geq 0}$ be a sequence of finite groups such that for all $i\geq 0$
    \begin{itemize}
        \item $G_{i+1}\leq G_{i}$;
        \item $G_k=0$ for some $k\geq 0$.
    \end{itemize}
    We call such $\cG$ a finite group filtration.
\end{definition}
\begin{notation}
    For a finite group filtration $\cG=(G_i)_{i\geq 0}$, and finite dimensional complex representations $\tau,\sigma$ of $G_0$, we define the following notation.  For $i\geq 0$:
    \begin{align*}
        \epsg_i(\tau)&=\abs{\tau}-\abs{\tau^{G_i}};
        &\delg_i(\tau,\sigma)&=\abs{\braces{\tau\otimes\sigma}^{G_i}}-\abs{\tau^{G_i}\otimes\sigma^{G_i}}.
    \end{align*}

\end{notation}
\begin{definition}  Let $\cG=(G_i)_{i\geq 0}$ be a finite group filtration, and $\tau,\sigma$ be representations of $G_0$.  The conductor exponent $\aG(\tau)$ and the function $\delg(\tau,\sigma)$ are defined by:
    \begin{align*}
        \aG(\tau)&=\sum_{i\geq 0}\frac{\epsg_i(\tau)}{[G_0:G_i]};&
        \delg(\tau,\sigma)&=\sum_{i\geq 0}\frac{\delg_i(\tau,\sigma)}{[G_0:G_i]}.
    \end{align*}
\end{definition}
\begin{rem}
    Note that $\epsg_i(\tau)$ and $\delg_i(\tau,\sigma)$ are decreasing sequences in $i$.
\end{rem}
In what follows we will prove results about $\aG$.  Our arithmetic motivations are to study two particular cases, which we outline briefly below.  
\begin{example}[Artin conductor]\label{example:artinconductor}
    Let $F$ be a non-Archimedean local field, and $\tau,\sigma$ be Weil representations over $F$.  Let $L/F^{\nr}$ be the finite extension through which $\tau$ and $\sigma$ factor, and $G_i\leq \gal(L/F^{\nr})$ be the $i$th ramification group in the lower numbering.  Then we have an equality
    \[\aG(\tau)=a(\tau),\]
    where the right hand side is the usual Artin conductor exponent, with the finite group filtration given by $\cG=\braces{G_{i}}_{i\geq 0}$.  Moreover, we have
    \[\delg(\tau,\sigma)\geq \delg_0(\tau,\sigma)=\deg(\tau\otimes\sigma)-\deg(\tau)\deg(\sigma)\geq 0,\]
    which measures nontrivial growth in the degree under tensor product.
\end{example}

\begin{example}[Swan conductor]\label{example:swanconductor}
    Let $F$ be a non-Archimedean local field, and $\tau$ be a Weil representation over $F$.  Let $L/F^{\nr}$ be the finite extension through which $\tau$ factors, and $G_i\leq \gal(L/F^{\nr})$ be the $i$th ramification group in the lower numbering.  Consider the finite group filtration given by $\cG=\braces{G_{i+1}}_{i\geq 0}$.  Then we have an equality
    \[\aG(\tau)=\Sw(\tau),\]
    where the right hand side is the usual Swan conductor exponent.
\end{example}

\subsection{Conductor bounds for tensor products}
We now consider tensor products of representations with some additional structure.  In one case, we will consider symplectic representations, in the other we consider representations of $p$-groups with rationality properties.  To begin with we will need the following.

\begin{lemma}
\label{lemRGi}
Let $\cG=(G_i)_{i\geq 0}$ be a finite group filtration, and let $\tau_1,\tau_2$ be representations of $G_0$.  Then for all $i\geq 0$,
\[\abs{\tau_1}\epsg_i(\tau_2)+\abs{\tau_2}\epsg_i(\tau_1)-\epsg_i(\tau_1\otimes\tau_2) = \epsg_i(\tau_1)\epsg_i(\tau_2)+\delg_i(\tau_1,\tau_2) .\]
\end{lemma}
\begin{proof}
Since $\epsg_i(\tau_j)=\abs{\tau_j}-\abs{\tau_j^{G_i}}$, clearly the right hand side of the claimed equality is
\[\epsg_i(\tau_1)\epsg_i(\tau_2)=\abs{\tau_1} \abs{\tau_2} - \abs{\tau_1^{G_i}}\abs{\tau_2} - \abs{\tau_2^{G_i}}\abs{\tau_1} + \abs{\braces{\tau_1\otimes\tau_2}^{G_i}}.\] 
On the other hand, again unwinding the definition of the $\epsg_i$, the left hand side of the claimed equality is also
\begin{align*}
\abs{\tau_1}\epsg_i(\tau_2)+\abs{\tau_2}\epsg_i(\tau_1)-\epsg_i(\tau_1\otimes\tau_2)&=\abs{\tau_1}\abs{\tau_2} - \abs{\tau_1^{G_i}}\abs{\tau_2} - \abs{\tau_2^{G_i}}\abs{\tau_1} + \abs{\braces{\tau_1\otimes\tau_2}^{G_i}}.
\end{align*}
\end{proof}

We will also make use of an elementary lemma to glue together the various steps in the ramification filtration.
\begin{lemma}
\label{lemseq}
    Let $M$ be a positive real number, and $(a_i)_{i=1}^n,(b_i)_{i=1}^n\subseteq \RR_{\geq M}\cup \set{0}$ be two decreasing sequences. Let $d_1,d_2,...,d_n$ be positive real numbers. Then 
    \[\sum_{i=1}^n \frac{a_ib_i}{d_i} \geq M \min \left( \sum_{i=1}^n \frac{a_i}{d_i}, \,\, \sum_{i=1}^n \frac{b_i}{d_i} \right).\]
\end{lemma}

\begin{proof}
    First suppose that $a_nb_n \neq 0$ (so all $a_i$ and $b_i$ are at least $M$). Then for all $1 \leq i \leq n$, $\frac{a_ib_i}{d_i} \geq M\max\{ \frac{a_i}{d_i},\frac{b_i}{d_i} \}$ and the claim follows. 
    
    Now let $k$ be minimal such that $a_kb_k=0$ and suppose, without loss of generality, that $a_k=0$. This implies that $b_j\geq M$ for $j<k$ so 
\begin{eqnarray*}
    \sum_{i=1}^n \frac{a_ib_i}{d_i} &=& \sum_{i=1}^{k-1} \frac{a_ib_i}{d_i} \\
    &\geq& M\sum_{i=1}^{k-1} \frac{a_i}{d_i} \\
    &=& M\sum_{i=1}^{n} \frac{a_i}{d_i}.
\end{eqnarray*} 
\end{proof}
\subsubsection{Symplectic representations}
We now study tensor products of symplectic representations, the main result for which is the following.  
\begin{proposition}
\label{prop:gp filtration R ineq}
Let $\cG=(G_i)_{i\geq 0}$ be a finite group filtration, and $\tau_1, \tau_2$ be symplectic representations of $G_0$.  Then
\[\aG(\tau_1\otimes\tau_2) \leq \abs{\tau_1}\aG(\tau_2)+\abs{\tau_2}\aG(\tau_1)-2 \min\set{\aG(\tau_1), \aG(\tau_2)} - \delg(\tau_1,\tau_2).\]
\end{proposition}
First, we remind the reader of a standard result, for which we give a short proof in the interests of self-containment.
\begin{lemma}
\label{lemsymp}
    Let $\tau$ be a symplectic representation of a finite group $G$.  Then $\abs{\tau}-\abs{\tau^G}$ is even.
\end{lemma}
\begin{proof}
    Let $\omega$ be the associated non-degenerate alternating $G$-equivariant bilinear pairing on $\tau$.  Since $\omega$ is non-degenerate, $\abs{\tau}$ is even, and so it is sufficient to check that $\omega$ remains non-degenerate when restricted to $\tau^G$.  Let $a\in\tau^G$, and choose $b\in\tau$ such that $\omega(a,b)\neq 0$.  Then note that, since $\omega$ is $G$-equivariant, $\omega\braces{a, \sum_{g\in G}g\cdot b}=\braces{\#G}\omega\braces{a, b}\neq 0$.
    In particular, since $\sum_{g\in G}g\cdot b\in \tau^G$, the pairing remains non-degenerate on $\tau^G$.
\end{proof}
\noindent We are now ready to prove the bound for symplectic representations.

\begin{proof}[Proof of \Cref{prop:gp filtration R ineq}]
    Note that 
    \begin{align*}
    \abs{\tau_1}\aG(\tau_2)+\abs{\tau_2}\aG(\tau_1)-\aG(\tau_1\otimes\tau_2)
    &=\sum_{i\geq 0}\frac{\abs{\tau_1}\epsg_i(\tau_2)+\abs{\tau_2}\epsg_i(\tau_1)-\epsg_i(\tau_1\otimes\tau_2)}{[G_0:G_i]}.
    \end{align*}
    Thus, by \Cref{lemRGi}, it suffices to prove
    \begin{equation*}
    \sum_{i \geq 0} \frac{\epsg_i(\tau_1)\epsg_i(\tau_2)}{[G_0:G_i]} \geq 2 \min_j \left( \sum_{i \geq 0} \frac{\epsg_i(\tau_j)}{[G_0:G_i]} \right).
    \end{equation*}
    By \Cref{lemsymp}, $\epsg_i(\tau_j) \in 2\ZZ_{\geq 0}$, hence this follows from \Cref{lemseq}.
\end{proof}

\subsubsection{Representations of \texorpdfstring{$p$}{p}-groups with rational characteristic polynomial}
We now consider tensor products of representations of $p$-groups which have rational characteristic polynomial.  The main result here is the following.
\begin{proposition}
    \label{prop:gp filtration p-group inequality}
Let $p$ be a prime number, and $\cG=(G_i)_{i\geq 0}$ be a finite group filtration such that $G_0$ is a $p$-group.  Let $\tau_1, \tau_2$ be representations of $G_0$ such that for every $g\in G_0$ and every $j\in\set{1,2}$ the characteristic polynomial of $\tau_j(g)$ has coefficients in $\QQ$.  Then
\[\aG(\tau_1\otimes\tau_2) \leq \abs{\tau_1}\aG(\tau_2) + \abs{\tau_2}\aG(\tau_1)-(p-1) \min\set{\aG(\tau_1), \aG(\tau_2)} - \delg(\tau_1,\tau_2).\]
\end{proposition}
We first remind the reader of an important standard result, for which we present a short proof.
\begin{lemma}
\label{lem:pgrouprationalreps}
    Let $p$ be a prime number and $G$ be a finite $p$-group.  Let $\tau$ be a representation of $G$, and assume that for all $g\in G$ the characteristic polynomial of $\tau(g)$ has coefficients in $\QQ$.  Then 
    \[\abs{\tau}-\abs{\tau^G}\in \set{0}\cup \ZZ_{\geq (p-1)}.\]
\end{lemma}
\begin{proof}
    If $\tau=\tau^G$ then the result is apparent, so we assume that there is a nontrivial irreducible subrepresentation $\chi$ of $\tau$. Since $\abs{\chi}\mid \#G$, and so is a power of $p$, if $\chi$ is not $1$-dimensional then the result follows.  Hence it remains to consider the case that $\chi$ is a $1$-dimensional representation of the $p$-group $G$.  Since $\chi$ is a character of order a power of $p$, it has at least $p-1$ conjugate representations, and since the characteristic polynomials of all elements are rational these conjugates must all be subrepresentations of $\tau$, so again the result follows.
\end{proof}
\noindent Our proposition now follows.
\begin{proof}[Proof of \Cref{prop:gp filtration p-group inequality}]
    By \Cref{lemRGi}, we know that for all $i\geq 0$
    \[\aG(\tau_1\otimes\tau_2) = \abs{\tau_1}\aG(\tau_2)+\abs{\tau_2}\aG(\tau_1) - \sum_{i\geq 0}\frac{\aG_i(\tau_1)\aG_i(\tau_2)}{[G_0:G_i]}-\delg(\tau_1,\tau_2),\]
    so it is equivalent to prove the claim that $\sum_{i\geq 0}\frac{\aG_i(\tau_1)\aG_i(\tau_2)}{[G_0:G_i]}\geq (p-1)\min\set{\aG(\tau_1),\aG(\tau_2)}$, which follows from \Cref{lem:pgrouprationalreps} and \Cref{lemseq}.
\end{proof}

\section{Weil--Deligne Representations}\label{sec:arithmetic}
We begin this section by reminding the reader of some standard results on Weil--Deligne representations, which we will use throughout.  We then go on to prove our main local results from the introduction, and deduce the global conductor results from these at the end.
\subsection{Background}\label{subsec:WD reps background}
Recall that (Frobenius-semisimple) Weil--Deligne representations are isomorphic to direct sums of representations of the form $\sigma_n\otimes \sp(n)$ \cite{Tate}*{4.1.5}, where $\sp(n)$ is the special representation \cite{Tate}*{4.1.4} and $\sigma_n$ is a Weil representation.  For abelian varieties, it is well-known that the corresponding Weil--Deligne representation is of a specific shape.
\begin{proposition}[for example, \cite{Sab}*{Proposition 1.10}, \cite{D2LocalGal}*{Corollary 6}]
\label{prop:abvarrep}
    Let $A/F$ be an abelian variety over a non-Archimedean local field, and $\rho_A$ the corresponding Weil--Deligne representation. Then there exists a(n essentially) symplectic Weil representation $\tau$ and self-dual Artin representation $\sigma$ such that 
    \[\rho_A \cong \tau \oplus (\sigma \otimes \sp(2)).\]
    Moreover, denoting the inertia subgroup by $G_0\leq G_F$, for every $g\in G_0$ the characteristic polynomial of $\rho_A(g)$ has coefficients in $\QQ$.
\end{proposition}

\noindent We use the following result to compute conductor exponents for twists of the special representations.
\begin{lemma}[\cite{Rohr}*{Proposition p.141}]
\label{spcond}
    Let $\sigma$ be a Weil representation over a non-Archimedean local field. Then 
    \[a(\sigma \otimes \sp(n)) = na(\sigma) + (n{-}1)\abs{\sigma^{G_0}},\]
    where $G_0\leq G_F$ is the inertia subgroup.
\end{lemma}
We combine these results, together with standard results for tensor products of special representations, to note the following useful result.

\begin{lemma}
\label{lem:conductor of TP of spn spm}  Let $\sigma$ be a Weil representation over a non-Archimedean local field, and $m,n\in\ZZ_{\geq 1}$ be positive integers, then
    \[a\braces{\sigma\otimes\sp(n)\otimes\sp(m)}=nm\cdot a(\sigma)+\abs{\sigma^{G_0}}\braces{nm-\min\set{n,m}},\]
    where $G_0\leq G_F$ is the inertia subgroup.
\end{lemma}
\begin{proof}
    It is a standard result that $\sp(n)\otimes\sp(m)\cong \bigoplus_{i=1}^r\psi_i\otimes \sp(k_i)$, for some unramified Weil characters $\psi_i$, and note that $\sum_{i=1}^rk_i=nm$, and $r=\min\set{n,m}$ (see for example \cite{BushHen2}*{Proposition 2.5 and proof}.  Thus, using \Cref{spcond}
    \begin{align*}
        a(\sigma\otimes\sp(n)\otimes\sp(m))&=\sum_{i=1}^ra\braces{\sigma\otimes\sigma_i\otimes \sp(k_i)}
        \\&=\sum_{i=1}^r\braces{k_ia(\sigma)+\braces{k_i-1}\abs{\sigma^{G_0}}}
        \\&=nm\cdot a(\sigma)+\abs{\sigma^{G_0}}\braces{nm-\min\set{n,m}}.
    \end{align*}
\end{proof}

\subsection{Conductor exponents for tensor products}\label{subsec:WD reps conductor exponents for tps}

We now consider tensor products of representations associated to abelian varieties.  We begin with the case that one of the varieties is semistable, where we write an exact equality for the conductor exponent of the tensor product.  Following this we provide a bound for the general case, where our approach will be to split the conductor exponent into the tame part and wild part (i.e. the Swan conductor).  We deduce the min inequality for the wild part from \Cref{prop:gp filtration R ineq}, since the action of wild inertia factors through a finite quotient, and for the tame part we prove a stronger max inequality.  There are two families of special cases where the max inequality fails to hold but those are contained in the semistable case.
\subsubsection{The semistable case}
We begin by computing the exact conductor exponent in the case that one abelian variety is semistable.
\begin{theorem}\label{thm:semistable case}
Let $A$ and $B$ be abelian varieties over a non-Archimedean local field $F$ and suppose that $A$ is semistable. Then
\[a(\rho_A \otimes \rho_B) = \abs{\rho_A}a(\rho_B) + \deg(B)a(\rho_A) - \braces{\deg(A\boxtimes B)-\deg(A)\deg(B)}.\]
\end{theorem}

\begin{proof}
    Let $G_0\leq G_F$ be the inertia subgroup.  By assumption, we can write $\Res_{G_0}\rho_A \cong \triv^m  \oplus \sp(2)^n$ for some $n,m\in\ZZ_{\geq 0}$.  Moreover, using \Cref{prop:abvarrep} we write $\rho_B= \tau \oplus (\sigma \otimes \sp(2))$ for two Weil representations $\sigma,\tau$.  

    From this description, via \Cref{spcond} and \Cref{lem:conductor of TP of spn spm}, we compute:
    \begin{align*}
    a(\rho_A) &= n, \\
    a(\rho_B) &= a(\tau)+2a(\sigma)+\abs{\sigma^{G_0}},\\
    a(\rho_A \otimes \rho_B) &=ma(\tau)+2na(\tau)+n\abs{\tau^{G_0}}+2ma(\sigma)+m\abs{\sigma^{G_0}}+n\braces{4a(\sigma)+2\abs{\sigma^{G_0}}}\\
    &= n\abs{\tau^{G_0}}+ \abs{\rho_A}a(\rho_B).
    \end{align*}
    On the other hand, by definition
    \begin{align*}
    \deg(A)=\abs{\rho_A^{G_0}} &= m+n, \\
    \deg(B)=\abs{\rho_B^{G_0}} &= \abs{\tau^{G_0}}+ \abs{\sigma^{G_0}}, \\
    \deg(A \boxtimes B)=\abs{\braces{\rho_A\otimes\rho_B}^{G_0}} &= m\braces{\abs{\tau^{G_0}}+\abs{\sigma^{G_0}}}+n\braces{\abs{\tau^{G_0}}+2\abs{\sigma^{G_0}}}\\
    &=n\abs{\sigma^{G_0}} + (m+n)(\abs{\tau^{G_0}} + \abs{\sigma^{G_0}}).
    \end{align*}
    From this it is then clear that
    \begin{align*}
    \abs{\rho_A}a(\rho_B)& + \deg(B)a(\rho_A) - \braces{\deg(A\boxtimes B)-\deg(A)\deg(B)}\\
    &=\abs{\rho_A}a(\rho_B)+n\braces{\abs{\tau^{G_0}}+\abs{\sigma^{G_0}}}-n\abs{\sigma^{G_0}}\\
    &=a(\rho_{A}\otimes\rho_{B}),
    \end{align*}
    as required.
\end{proof}

\subsubsection{The general case}

We now consider the situation where $A$ and $B$ are arbitrary abelian varieties.

\begin{proposition}\label{thm:swancond of tensor of abvars}
Let $A$ and $B$ be abelian varieties over a non-Archimedean local field $F$ of residue characteristic $p$.  Then 
\[\Sw(\rho_A \otimes \rho_B)\leq \abs{\rho_A}\Sw(\rho_B) + \abs{\rho_B}\Sw(\rho_A) - \max\set{2, p-1}\min \set{\Sw(\rho_A), \Sw(\rho_B)}.\]
\end{proposition}
\begin{proof}
Let $G_0\leq G_F$ be the inertia subgroup. By \Cref{prop:abvarrep} we have
\begin{align*}
    \rho_1:=\rho_A&=\tau_1\oplus\braces{\sigma_1\otimes \sp(2)},\\
    \rho_2:=\rho_B&=\tau_2\oplus\braces{\sigma_2\otimes \sp(2)},
\end{align*}
where for $j\in\set{1,2}$, $\Res_{G_0}\tau_j$ is a symplectic Artin representation and $\Res_{G_0}\sigma_j$ is a self-dual Artin representation.  Let $\tilde{\rho_j}=\tau_j\oplus\sigma_j^2$, note that since $\Res_{G_0}\tilde{\rho_j}\cong \Res_{G_0}{\rho_j}$, we have an equality $\Sw(\rho_j)=\Sw(\tilde{\rho_j})$, and so we work here instead.

We write $G=\gal(L/F^{\nr})$ for a finite Galois group through which the representations $\tilde{\rho_j}$ both factor, and let $G_i\leq G$ be the $i$th ramification group in the lower numbering.  We take the finite group filtration $\cG=\braces{G_{i+1}}_{i\geq 0}$, for which the quantity $\epsg$ is the Swan conductor (see also \Cref{example:swanconductor}).  If $p>2$, then as $G_1$ is a $p$-group and by \Cref{prop:abvarrep} the representations $\tilde{\rho_j}$ satisfy the conditions of \Cref{prop:gp filtration p-group inequality}, the result follows.  If instead $p=2$, then similarly the representations $\tilde{\rho_j}$ are symplectic, so we instead apply \Cref{prop:gp filtration R ineq}.
\end{proof}

\begin{rem}
    We ignore the possible improvement to our Swan conductor bound above coming from the $\delg$ term in \Cref{prop:gp filtration R ineq,prop:gp filtration p-group inequality}, since these terms do not clearly have a natural arithmetic interpretation.
\end{rem}

\begin{lemma}\label{lem:tensorprod tame part}
    Let $A$ and $B$ be abelian varieties over a non-Archimedean local field $F$.  Assume that both $a(\rho_A),a(\rho_B)>1$.
    Then
    \[\abs{\frac{\rho_A\otimes\rho_B}{\braces{\rho_A\otimes\rho_B}^{G_0}}}
    = \abs{\rho_A}\abs{\frac{\rho_B}{\rho_B^{G_0}}}+\abs{\rho_B}\abs{\frac{\rho_A}{\rho_A^{G_0}}} - \abs{\frac{\rho_A}{\rho_A^{G_0}}}\abs{\frac{\rho_B}{\rho_B^{G_0}}}-\bigg(\deg(A\boxtimes B)-\deg(A)\deg(B)\bigg)\]
\end{lemma}
\begin{proof}
    We rearrange the right hand side of the claimed inequality, to obtain
    \begin{align*}
        &\abs{\rho_A}\abs{\frac{\rho_B}{\rho_B^{G_0}}}+\abs{\rho_B}\abs{\frac{\rho_A}{\rho_A^{G_0}}} - \abs{\frac{\rho_A}{\rho_A^{G_0}}}\abs{\frac{\rho_B}{\rho_B^{G_0}}}-\bigg(\deg(A\boxtimes B)-\deg(A)\deg(B)\bigg)
        \\&=\abs{\rho_A}\abs{\rho_B}-\abs{\rho_A^{G_0}}\abs{\rho_B^{G_0}}-
        \bigg(\abs{\braces{\rho_A\otimes\rho_B}^{G_0}}-\abs{\rho_A^{G_0}}\abs{\rho_B^{G_0}}\bigg)
        \\&=\abs{\frac{\rho_A\otimes\rho_B}{\braces{\rho_A\otimes\rho_B}^{G_0}}}.
    \end{align*}
\end{proof}

\begin{rem}
    If $\rho_A$ and $\rho_B$ are both tame representations, which necessarily occurs when $\dim(A),\dim(B)$ are strictly less than $(p-1)/2$ for the residue characteristic $p$, then this immediately presents the conductor exponent of the tensor product in terms of the same inputs we will bound with later.
\end{rem}

We now stitch together the tame and wild parts of the conductor to obtain the following theorem.

\begin{theorem}[\Cref{INTRO:thm:cond of tensor of abvars}]\label{thm:tensors of abvars}Let $A$ and $B$ be abelian varieties over a non-Archimedean local field $F$ of residue characteristic $p$.  Define the constant
    \[C_p(A,B):=\min\bigg\{\max\set{2,p-1},\ \abs{\rho_A} - \deg(A),\ \abs{\rho_B}-\deg(B)\bigg\}\]
Then
    \[a(\rho_A \otimes \rho_B)\leq \abs{\rho_A}a(\rho_B) +\abs{\rho_B}a(\rho_A)- C_p(A,B)\min \set{a(\rho_A), a(\rho_B)} - \bigg(\deg(A\boxtimes B)-\deg(A)\deg(B)\bigg).\]
    Moreover, if $a(\rho_A)\leq 1$ then in fact
    \[a(\rho_A \otimes \rho_B) = \abs{\rho_A}a(\rho_B) + \deg(B)a(\rho_A) - \bigg(\deg(A\boxtimes B)-\deg(A)\deg(B)\bigg).\]
\end{theorem}
\begin{proof}
We begin with the case $a(\rho_A)\leq 1$.  Since $a(\rho_A)\leq 1$, by \Cref{spcond} $A$ is semistable, so this is a consequence of \Cref{thm:semistable case}.  It remains in this case to show that the claimed equality is within the proposed bound.  In other words we must show that $\abs{\rho_B/\rho_B^{G_0}}a(\rho_A)-C_p(A,B)\min\set{a(\rho_A),a(\rho_B)}\geq 0$.  Indeed, if $a(\rho_B)=0$ then clearly this is $0$, and else we have
\[\abs{\rho_B/\rho_B^{G_0}}a(\rho_A)-C_p(A,B)\min\set{a(\rho_A),a(\rho_B)}=\abs{\rho_B/\rho_B^{G_0}}-1\geq 0.\]

Now, by symmetry, we proceed under the assumption that both $a(\rho_A),a(\rho_B)>1$.  In particular, it follows from \Cref{lem:tensorprod tame part} and \Cref{thm:swancond of tensor of abvars} that
\begin{align*}
    a(\rho_A\otimes\rho_B) =& \abs{\frac{\rho_A\otimes\rho_B}{\braces{\rho_A\otimes\rho_B}^{G_0}}}+\Sw(\rho_A\otimes\rho_B)\\
    \leq& 
     \abs{\rho_A}a(\rho_B)+\abs{\rho_B}a(\rho_A)
    -\bigg(\deg(A\boxtimes B)-\deg(A)\deg(B)\bigg)\\
    &-\bigg(\abs{\frac{\rho_A}{\rho_A^{G_0}}}\abs{\frac{\rho_B}{\rho_B^{G_0}}} + \max\set{2, p-1}\min \set{\Sw(\rho_A), \Sw(\rho_B)}\bigg).
\end{align*}
Now, since
\begin{align*}
    &\abs{\frac{\rho_A}{\rho_A^{G_0}}}\abs{\frac{\rho_B}{\rho_B^{G_0}}} + \max\set{2, p-1}\min \set{\Sw(\rho_A), \Sw(\rho_B)}
    \\&\geq C_p(A,B)\bigg(\max\set{\abs{\frac{\rho_A}{\rho_A^{G_0}}},\abs{\frac{\rho_B}{\rho_B^{G_0}}}}+\min\set{\Sw(\rho_A),\Sw(\rho_B)}\bigg)
    \\&\geq C_p(A,B)\min\set{a(\rho_A),\ a(\rho_B)},
\end{align*}
the result follows.
\end{proof}

This bound requires a reasonable amount of information to compute, and so we provide a weaker, but more uniform, bound as a corollary below.
\begin{corollary}\label{cor:simplified conductor exp bound}
    Let $A$ and $B$ be abelian varieties over a non-Archimedean local field $F$ of residue characteristic $p$.  Then
    \[a(\rho_A\otimes \rho_B)\leq \abs{\rho_A}a(\rho_B)+\abs{\rho_B}a(\rho_A)-2\min\set{a(\rho_A),a(\rho_B)}-\Delta(A,B),\]
    where
    \[\Delta(A,B):=\begin{cases}
        \deg(A\boxtimes B)-\deg(A)\deg(B)&\textnormal{if }a(\rho_A)a(\rho_B)>1,\\
        0&\textnormal{else.}
    \end{cases}\]    
\end{corollary}
\begin{proof}
    This is immediate from \Cref{thm:tensors of abvars} so long as $C_p(A,B)\geq 2$, and so it remains to show this precisely when (without loss of generality) $\abs{\rho_A}-\deg(\rho_A)\leq 1$.  In this case, it is clear from \Cref{spcond} and \Cref{prop:abvarrep} that in fact $a(\rho_A)\leq 1$.  In this case, writing $G_0\leq G_F$ for the inertia subgroup, it follows from \Cref{thm:tensors of abvars} that
    \begin{align*}
    a(\rho_A \otimes \rho_B) &= \abs{\rho_A}a(\rho_B) + \deg(B)a(\rho_A) - \bigg(\deg(A\boxtimes B)-\deg(A)\deg(B)\bigg)\\
    &=\abs{\rho_A}a(\rho_B)+\abs{\rho_B}a(\rho_A)-\abs{\rho_B/\rho_B^{G_0}}a(\rho_A)-\bigg(\deg(A\boxtimes B)-\deg(A)\deg(B)\bigg).
    \end{align*}
    In particular, it suffices to prove the new claim that
    \[\abs{\rho_B/\rho_B^{G_0}}a(\rho_A)+\bigg(\deg(A\boxtimes B)-\deg(A)\deg(B)\bigg)\geq 2\min\set{a(\rho_A),a(\rho_B)}+\Delta(A,B).\]
    If $a(\rho_A)=0$ then this is immediate, so via the symmetric argument for $B$ we are left with the case that $a(\rho_A)=1$ and $a(\rho_B)\geq 1$.  If $a(\rho_B)=1$, then both sides are equal to $2$; else $a(\rho_B)>1$, in which case $\abs{\rho_B/\rho_B^{G_0}}\geq 2$ and so the claim is again clear.
\end{proof}

We now remark on a useful bound for the difference in the degrees in the case of a single abelian variety.
\begin{lemma}\label{lem:degree difference for single abvar}
Let $A$ be an abelian variety over a non-Archimedean local field $F$.  Suppose $a(\rho_A)>0$.  Then
    \[\deg\braces{\rho_A\otimes\rho_A}-\deg(\rho_A)^2\geq 1.\]
\end{lemma}
\begin{proof}
    By \Cref{prop:abvarrep}, we write $\rho_A=\tau\oplus \braces{\sigma\otimes \sp(2)}$, where $\sigma,\tau$ are self-dual Weil representations.  Write $G:=\gal(L/F^{\nr})$ for a finite Galois group through which these representations factor, $G_0$ for the inertia subgroup, and $\Irr(G_0)$ for the set of irreducible (complex) characters of $G_0$.  Firstly, we identify
    \begin{align*}
        \deg(\rho_A)&=\abs{\tau^{G_0}}+\abs{\sigma^{G_0}}\\
        &=\gp{\tau,\triv}_{G_0}+\gp{\sigma,\triv}_{G_0}.
    \end{align*}
    Now, by \Cref{lem:conductor of TP of spn spm},
    \begin{align*}
    \deg(\rho_A\otimes\rho_A)&=\abs{(\tau\otimes\tau)^{G_0}}+2\abs{(\tau\otimes\sigma)^{G_0}}+2\abs{(\sigma\otimes\sigma)^{G_0}}\\
    &=\gp{\tau,\tau}_{G_0}+2\gp{\tau,\sigma}_{G_0}+2\gp{\sigma,\sigma}_{G_0}\\
    &=\sum_{\chi\in\Irr(G_0)}\gp{\tau,\chi}_{G_0}^2+2\gp{\sigma,\chi}_{G_0}^2+2\gp{\sigma,\chi}_{G_0}\gp{\tau,\chi}_{G_0}\\
    &=\sum_{\chi\in\Irr(G_0)}\braces{\gp{\tau,\chi}_{G_0}+\gp{\sigma,\chi}_{G_0}}^{2}+\gp{\sigma,\chi}_{G_0}.
    \end{align*}
    In particular,
    \[\deg(\rho_A\otimes\rho_A)-\deg(\rho_A)^2=\sum_{\substack{\chi\in\Irr(G_0)\\\chi\neq\triv}}\braces{\gp{\tau,\chi}_{G_0}+\gp{\sigma,\chi}_{G_0}}^{2}+\sum_{\chi\in\Irr(G_0)}\gp{\sigma,\chi}_{G_0}.\]
    Since $a(\rho_A)>0$, by \Cref{spcond} we note that either $\tau\neq \tau^{G_0}$ or $\sigma\neq 0$.  In particular this difference is at least $1$.
\end{proof}

\subsection{Global results}\label{subsec:global results}
Having shown our local results, we conclude by deducing the global results.

\begin{corollary}[\Cref{INTRO:thm:Cond of TP}]\label{cor:Cond of TP}
    Let $A,B$ be abelian varieties over a global field $K$ of conductors $\fN_A$ and $\fN_B$. Define
    \[\fd(A,B):=\prod_{\substack{\fp\in\finplaces_K\\v_\fp(\fN_A)v_\fp(\fN_B)>1}}\fp^{\deg_\fp(A\boxtimes B)-\deg_\fp(A)\deg_\fp(B)}.\]
    Then 
    \[N(A\boxtimes B)\hspace{2pt}\Bigg\vert\hspace{2pt} \frac{\fN_A^{2\dim B}\fN_B^{2\dim A}}{\fd(A,B)\gcd(\fN_A,\fN_B)^2},\]
    where $N(A\boxtimes B)$ is the conductor of the Rankin--Selberg $L$-function $L(A\boxtimes B, s)$.
\end{corollary}
\begin{proof}
    Immediate from \Cref{cor:simplified conductor exp bound}.
\end{proof}

\begin{corollary}\label{cor:cond of self TP}
    Let $A$ be an abelian variety over a global field $K$ of conductor $\fN_A$.  Moreover, write $\fN_{A,2}:=\prod\limits_{\substack{\fp\in\finplaces_K, \\v_{\fp}(\fN_A)\geq 2}}\fp$.  Then
    \[N(A\boxtimes A)\hspace{2pt}\Bigg\vert\hspace{2pt} \frac{\fN_A^{4\dim A - 2}}{\fN_{A,2}},\]
    where $N(A\boxtimes A)$ is the conductor of the Rankin--Selberg $L$-function $L(A\boxtimes A, s)$.
\end{corollary}
\begin{proof}
    It follows from \Cref{cor:Cond of TP} that $N(A\boxtimes A)\mid \frac{\fN_A^{4\dim A - 2}}{\fd(A,A)}$, and so it is sufficient to show that $\fN_{A,2}\mid \fd(A,A)$.  This is immediate from \Cref{lem:degree difference for single abvar}.
\end{proof}

\subsection{Sharpness: a family of examples}\label{subsec:sharpness}
The purpose of this subsection to provide a family of nontrivial examples of sharpness for each $p$ for the inequalities of \Cref{INTRO:thm:cond of tensor of abvars,INTRO:thm:swancond of tensor of abvars}. Our examples will be Jacobians $J_\alpha$, indexed by $\alpha\in\ZZ_p$, of the curves
\[C_{\alpha}/\QQ_p: y^2=x^p-\alpha.\]

\begin{proposition}
\label{thm:sharpness1}
    Let $p$ be an odd prime number, and $a \in \ZZ_p^{\times}$ such that $a^{p-1} \not\equiv 1 \bmod{p^2}$.  Then the inequalities in \Cref{INTRO:thm:cond of tensor of abvars,INTRO:thm:swancond of tensor of abvars} are equalities when $A=J_a, B=J_p$.
\end{proposition}

\begin{lemma}
\label{lem:inertiaexample1}
    Let $p$ be an odd prime, and $a\in\ZZ_p^\times$ such that $a^{p-1}\not\equiv 1 \bmod{p^2}$. Then $\Sw(\rho_{J_a})=1$ and $\Sw(\rho_{J_p})=p$.
\end{lemma}

\begin{proof}
    We compute the Swan exponent using the formula in \cite{MR4453743}*{Theorem 12.3}; namely, whenever $\alpha \not\in (\QQ_p^{\times})^p$ it is given by 
    \[\Sw(J_{\alpha})=v(\Delta_{\QQ_p(\sqrt[p]{\alpha})/\QQ_p}) - [\QQ_p(\sqrt[p]{\alpha}) : \QQ_p] + f_{\QQ_p(\sqrt[p]{\alpha})/\QQ_p},\]
    where $\Delta, f$ are the discriminant and residue degree respectively.  As $x^p-p$ and $(x+a+p)^p-a$ are Eisenstein, for both $\alpha\in\set{a,p}$ the extension $\QQ_p(\sqrt[p]{\alpha})/\QQ_p$ is totally ramified of degree $p$ with ring of integers $\ZZ_p[\sqrt[p]{\alpha}]$.  In particular, $\Delta_{\QQ_p(\sqrt[p]{\alpha})/\QQ_p}=\pm \alpha^{p-1}p^p$, and the claim follows.
\end{proof}

\begin{proof}[Proof of \Cref{thm:sharpness1}]
We begin with the Swan conductor.  It is immediate from \Cref{lem:inertiaexample1} that \[\abs{\rho_{J_{a}}}\Sw(\rho_{J_{p}}) + \abs{\rho_{J_{p}}}\Sw(\rho_{J_{a}}) - \max\{2,p{-}1\}\min\{\Sw(\rho_{J_{a}}),\Sw(\rho_{J_{p}})\}=p(p-1),\]
so we will now compute $\Sw(\rho_{J_a}\otimes\rho_{J_p})$ and show this is the same.  As the curves $C_{\alpha}$ are hyperelliptic, the image of the wild inertia group under $\rho_{J_{\alpha}}$ is isomorphic to the wild inertia group of the Galois group of $x^p-\alpha$ which is precisely $\FF_p$ for $\alpha \in \set{a,p}$. Since they have rational characteristic polynomial on $G_0$ and dimension $p-1$, we must have
\begin{align*}
    \Res_{G_1}\rho_{J_a}&\cong \bigoplus_{i=1}^{p-1}\chi_a^{\otimes i}& \text{and } \Res_{G_1}\rho_{J_p}\cong \bigoplus_{i=1}^{p-1}\chi_p^{\otimes i},
\end{align*}
where $\chi_\alpha$ is a choice of nontrivial Artin representation over $\QQ_p(\zeta_p)$ with dimension $1$ factoring through $\QQ_p(\sqrt[p]{\alpha},\zeta_p)$.  We then directly compute from the definition that $\Sw(\rho_{J_a}\otimes\rho_{J_p})=p(p-1)$, as required.

For the Artin conductor, when $G_1$ has no fixed points nor does $G_0$ and so it is easy to compute from the above that for $\alpha\in\set{a,p}$, we have $a(\rho_{J_\alpha})=(p-1)+\Sw(\rho_{J_\alpha})$, and that the degree terms are all $0$, so
\begin{align*}
&\abs{\rho_{J_a}}a(\rho_{J_p}) + \abs{\rho_{J_p}}a(\rho_{J_a}) - C_p(J_a,J_p)\min\{a(\rho_{J_a}),a(\rho_{J_p})\} - (\deg(J_a \boxtimes J_p) - \deg(J_a)\deg(J_p))
\\&=p(p-1)+(2p-1)(p-1)-(p-1)p-0
\\&=(2p-1)(p-1).
\end{align*}
Moreover, $a(\rho_{J_a}\otimes\rho_{J_p})=(p-1)^2+\Sw(\rho_{J_a}\otimes\rho_{J_p})=(2p-1)(p-1)$ as required.
\end{proof}
\newpage
\bibliography{refs}

\begin{bibdiv}
\begin{biblist}

\bib{MR4453743}{article}{
      author={Best, Alex~J.},
      author={Betts, L.~Alexander},
      author={Bisatt, Matthew},
      author={van Bommel, Raymond},
      author={Dokchitser, Vladimir},
      author={Faraggi, Omri},
      author={Kunzweiler, Sabrina},
      author={Maistret, C\'{e}line},
      author={Morgan, Adam},
      author={Muselli, Simone},
      author={Nowell, Sarah},
       title={A user's guide to the local arithmetic of hyperelliptic curves},
        date={2022},
        ISSN={0024-6093,1469-2120},
     journal={Bull. Lond. Math. Soc.},
      volume={54},
      number={3},
       pages={825\ndash 867},
         url={https://doi.org/10.1112/blms.12604},
      review={\MR{4453743}},
}

\bib{BushHen2}{article}{
      author={Bushnell, C.~J.},
      author={Henniart, G.},
       title={Strong exponent bounds for the local {R}ankin-{S}elberg
  convolution},
        date={2017},
        ISSN={1017-060X,1735-8515},
     journal={Bull. Iranian Math. Soc.},
      volume={43},
      number={4},
       pages={143\ndash 167},
      review={\MR{3711826}},
}

\bib{MR2275908}{article}{
      author={Brumley, Farrell},
       title={Effective multiplicity one on {${\rm GL}_N$} and narrow zero-free
  regions for {R}ankin-{S}elberg {$L$}-functions},
        date={2006},
        ISSN={0002-9327,1080-6377},
     journal={Amer. J. Math.},
      volume={128},
      number={6},
       pages={1455\ndash 1474},
  url={http://muse.jhu.edu/journals/american_journal_of_mathematics/v128/128.6brumley.pdf},
      review={\MR{2275908}},
}

\bib{MR4404782}{article}{
      author={Brumley, Farrell},
      author={Thorner, Jesse},
      author={Zaman, Asif},
       title={Zeros of {R}ankin-{S}elberg {$L$}-functions at the edge of the
  critical strip},
        date={2022},
        ISSN={1435-9855,1435-9863},
     journal={J. Eur. Math. Soc. (JEMS)},
      volume={24},
      number={5},
       pages={1471\ndash 1541},
         url={https://doi.org/10.4171/jems/1134},
        note={With an appendix by Colin J. Bushnell and Guy Henniart},
      review={\MR{4404782}},
}

\bib{MR3307911}{incollection}{
      author={Cogdell, James~W.},
       title={{$L$}-functions and functoriality},
        date={2014},
   booktitle={Automorphic forms and {$L$}-functions},
      series={Adv. Lect. Math. (ALM)},
      volume={30},
   publisher={Int. Press, Somerville, MA},
       pages={1\ndash 47},
      review={\MR{3307911}},
}

\bib{D2LocalGal}{article}{
      author={Dokchitser, Tim},
      author={Dokchitser, Vladimir},
       title={Local galois representations and frobenius traces},
        date={2022},
     journal={arXiv:2201.04094},
         url={https://arxiv.org/pdf/2201.04094.pdf},
}

\bib{MR1728875}{article}{
      author={Duke, W.},
      author={Kowalski, E.},
       title={A problem of {L}innik for elliptic curves and mean-value
  estimates for automorphic representations},
        date={2000},
        ISSN={0020-9910,1432-1297},
     journal={Invent. Math.},
      volume={139},
      number={1},
       pages={1\ndash 39},
         url={https://doi.org/10.1007/s002229900017},
        note={With an appendix by Dinakar Ramakrishnan},
      review={\MR{1728875}},
}

\bib{Rohr}{incollection}{
      author={Rohrlich, David~E.},
       title={Elliptic curves and the {W}eil-{D}eligne group},
        date={1994},
   booktitle={Elliptic curves and related topics},
      series={CRM Proc. Lecture Notes},
      volume={4},
   publisher={Amer. Math. Soc., Providence, RI},
       pages={125\ndash 157},
         url={https://doi.org/10.1090/crmp/004/10},
      review={\MR{1260960}},
}

\bib{MR1395406}{article}{
      author={Rudnick, Ze\'{e}v},
      author={Sarnak, Peter},
       title={Zeros of principal {$L$}-functions and random matrix theory},
        date={1996},
        ISSN={0012-7094,1547-7398},
     journal={Duke Math. J.},
      volume={81},
      number={2},
       pages={269\ndash 322},
         url={https://doi.org/10.1215/S0012-7094-96-08115-6},
      review={\MR{1395406}},
}

\bib{Sab}{article}{
      author={Sabitova, Maria},
       title={Root numbers of abelian varieties},
        date={2007},
        ISSN={0002-9947,1088-6850},
     journal={Trans. Amer. Math. Soc.},
      volume={359},
      number={9},
       pages={4259\ndash 4284},
         url={https://doi.org/10.1090/S0002-9947-07-04148-7},
      review={\MR{2309184}},
}

\bib{Tate}{incollection}{
      author={Tate, J.},
       title={Number theoretic background},
        date={1979},
   booktitle={Automorphic forms, representations and {$L$}-functions ({P}roc.
  {S}ympos. {P}ure {M}ath., {O}regon {S}tate {U}niv., {C}orvallis, {O}re.,
  1977), {P}art 2},
      series={Proc. Sympos. Pure Math.},
      volume={XXXIII},
   publisher={Amer. Math. Soc., Providence, RI},
       pages={3\ndash 26},
      review={\MR{546607}},
}

\end{biblist}
\end{bibdiv}
\end{document}